\documentclass[12pt]{article}
\usepackage{amsmath,amscd,amsthm,amssymb}

 \newtheorem{thm}{Theorem}[section]
 \newtheorem{cor}[thm]{Corollary}
 \newtheorem{lem}[thm]{Lemma}
 
 \theoremstyle{definition}
 \newtheorem{defn}[thm]{Definition}
 \theoremstyle{remark}
 \newtheorem{rem}[thm]{Remark}
 
 \numberwithin{equation}{section}

\def\Rn{{\mathbb{R}^n}}

\def\a {\alpha}
\def\i{\infty}
\def\l {\lambda}

\def\L1loc{L_{\Phi}^{\rm loc}(\Rn)}

\def\dual{\,^{^{\complement}}\!}

\usepackage{color}

\newcommand{\es}{\mathop{\rm ess \; inf}\limits}

\begin{document}

\begin{center}
\LARGE Boundedness of Intrinsic Square Functions and their Commutators on Generalized Orlicz-Morrey Spaces
\end{center}

\

\centerline{\large Vagif S. Guliyev$^{a,b,}$\footnote{corresponding author: vagif@guliyev.com (V.S. Guliyev)
\\
{$^{a}$ Department of Mathematics, Ahi Evran University, Kirsehir, Turkey}
\\
{$^{b}$ Institute of Mathematics and Mechanics, Baku, Azerbaijan}
},  Fatih Deringoz$^{a,1}$}

\

\begin{abstract}
We study the boundedness of intrinsic square functions and their commutators on generalized Orlicz-Morrey spaces $M^{\Phi,\varphi}(\Rn)$. In all the cases the conditions for the boundedness are given either in terms of Zygmund-type integral inequalities on weights $\varphi(x,r)$ without assuming any  monotonicity property of  $\varphi(x,r)$ on $r$.
\end{abstract}

\

\section{Introduction}

The intrinsic square functions were first introduced by Wilson in \cite{Wilson1, Wilson2}. They are defined as follows. For $0<\alpha\leq 1$, let $C_{\alpha}$ be the family of functions $\phi : \Rn\rightarrow \mathbb{R}$ such that $\phi$'s support is contained in $\{x:|x|\leq 1\},\ \displaystyle \int\phi dx =0$, and for $x,\ x'\in \mathbb{R}^{n}$,
$$
|\phi(x)-\phi(x')|\leq|x-x'|^{\alpha}.
$$
For $(y,\ t)\in \mathbb{R}_{+}^{n+1}$ and $f\in L_{loc}^{1}(\mathbb{R}^{n})$ , set
$$
A_{\alpha}f(t,\ y)\equiv\sup_{\phi\in C_{\alpha}}|f*\phi_{t}(y)|,
$$
where $\displaystyle \phi_{t}(y)=t^{-n}\phi(\frac{y}{t})$ . Then we define the varying-aperture intrinsic square (intrinsic Lusin) function of $f$ by the formula
$$
G_{\alpha,\beta}(f)(x)=\left(\int\int_{\Gamma_{\beta}(x)}(A_{\alpha}f(t,y))^{2}\frac{dydt}{t^{n+1}}\right)^{\frac{1}{2}}
$$
where $\Gamma_{\beta}(x)=\{(y,\ t)\in \mathbb{R}_{+}^{n+1}:|x-y|<\beta t\}$. Denote $G_{\alpha,1}(f)=G_{\alpha}(f)$ .

This function is independent of any particular kernel, such as Poisson kernel. It dominates pointwise the classical square function(Lusin area integral) and its real-variable generalizations. Although the function $G_{\alpha,\beta}(f)$ is depend of kernels with uniform compact support, there is pointwise relation between $G_{\alpha,\beta}(f)$ with different $\beta$:
$$
G_{\alpha,\beta}(f)(x)\leq\beta^{\frac{3n}{2}+\alpha}G_{\alpha}(f)(x)\ .
$$
We can see details in \cite{Wilson1}.

The intrinsic Littlewood-Paley $\mathrm{g}$-function and the intrinsic $g_{\lambda}^{*}$ function are defined respectively by
$$
g_{\alpha}f(x)=\left(\int_{0}^{\infty}(A_{\alpha}f(y,t))^{2}\frac{dt}{t}\right)^{\frac{1}{2}},
$$
$$
g_{\lambda,\alpha}^{*}f(x)=\left(\int\int_{\mathbb{R}_{+}^{n+1}}\left(\frac{t}{t+|x-y|}\right)^{n\lambda}(A_{\alpha}f(y,t))^{2}\frac{dydt}{t^{n+1}}\right)^{\frac{1}{2}}.
$$
Let $b$ be a locally integrable function on $\mathbb{R}^{n}$. Setting
$$
A_{\alpha,b}f(t,y)\equiv\sup_{\phi\in C_{\alpha}}\left|\int_{\mathbb{R}^{n}}[b(x)-b(z)]\phi_{t}(y-z)f(z)dz\right|,
$$
the commutators are defined by
$$
[b,G_{\alpha}]f(x)=\left(\int\int_{\Gamma(x)}(A_{\alpha,b}f(t,y))^{2}\frac{dydt}{t^{n+1}}\right)^{\frac{1}{2}}
$$
$$
[b,g_{\alpha}]f(x)=\left(\int_{0}^{\infty}(A_{\alpha,b}f(t,y))^{2}\frac{dt}{t}\right)^{\frac{1}{2}}
$$
and
$$
[b,g_{\lambda,\alpha}^{*}]f(x)=\left(\int\int_{\mathbb{R}_{+}^{n+1}}\left(\frac{t}{t+|x-y|}\right)^{\lambda n}(A_{\alpha,b}f(t,y))^{2}\frac{dydt}{t^{n+1}}\right)^{\frac{1}{2}}
$$

Wilson \cite{Wilson1} proved that $G_{\alpha}$ is bounded on $L^{p}(\mathbb{R}^{n})$ for $1<p<\infty$ and $0<\alpha\leq 1$. After then, Huang and Liu \cite{HuLiu} studied the boundedness of intrinsic square functions on weighted Hardy spaces. Moreover, they characterized the weighted Hardy spaces by intrinsic square functions. In \cite{Wang2} and \cite{WangLiu}, Wang and Liu obtained some weak type estimates on weighted Hardy spaces. In \cite{Wang1}, Wang considered intrinsic functions and the commutators generated with BMO functions on weighted Morrey spaces.
In \cite{GulShuk2013}, Guliyev, Shukurov was proved the boundedness of intrinsic square functions and their commutators on generalized Morrey spaces. In \cite{GulIntrGWM}, Guliyev considered intrinsic functions and the commutators generated with BMO functions on generalized weighted Morrey spaces.
In \cite{LiNaYaZh}, Liang et al. studied the boundedness of these operators on Musielak-Orlicz Morrey spaces.

In this paper, we will consider $G_{\alpha},\ g_{\alpha},\ g_{\lambda,\alpha}^{*}$ and their commutators on generalized Orlicz-Morrey spaces. Note that, the Orlicz-Morrey spaces were introduced and studied by Nakai in \cite{Nakai0}. Also the boundedness of the operators of harmonic analysis on Orlicz-Morrey spaces see also, \cite{DerGulSam, GulDerArX, GulDerHasJIA, HasJFSA, Nakai1, Nakai2, SawSugTan}. Our definition of Orlicz-Morrey spaces (see \cite{DerGulSam}) is different from that of Nakai \cite{Nakai0} and Sawano et al. \cite{SawSugTan}.

By $A \lesssim B$ we mean that $A \le C B$ with some positive constant $C$ independent of appropriate quantities. If $A \lesssim B$ and $B \lesssim A$, we write $A\approx B$ and say that $A$ and $B$ are  equivalent. Everywhere in the sequel $B(x,r)$ stands for the ball in $\mathbb{R}^n$ of radius $r$ centered at $x$ and $|B(x,r)|$ be the Lebesgue measure of the ball $B(x,r)$ and $|B(x,r)|=v_n r^n$, where $v_n$ is the volume of the unit ball in $\Rn$.

\

\section{Preliminaries}

As is well known that Morrey \cite{Morrey} introduced the classical Morrey spaces to investigate the local behavior of solutions to second order elliptic partial differential
equations (PDE). We recall its definition as
\begin{equation*}
M^{p,\lambda}(\Rn) = \left\{ f \in L^p_{\rm loc}(\Rn) : \left\| f\right\|_{M^{p,\lambda}}: = \sup_{x \in \Rn, \; r>0 } r^{-\frac{\lambda}{p}} \|f\|_{L^{p}(B(x,r))} < \i \right\},
\end{equation*}
where $0 \le \lambda \le  n,$ $1\le p < \i$. $M^{p,\l}(\Rn)$ was an expansion of $L^p(\Rn)$ in the sense that $M^{p,0}(\Rn)=L^p(\Rn)$.

We recall the definition of Young functions.

\begin{defn}\label{def2} A function $\Phi : [0,+\infty) \rightarrow [0,\infty]$ is called a Young function if $\Phi$ is convex, left-continuous, $\lim\limits_{r\rightarrow +0} \Phi(r) = \Phi(0) = 0$ and $\lim\limits_{r\rightarrow +\infty} \Phi(r) = \infty$.
\end{defn}

From the convexity and $\Phi(0) = 0$ it follows that any Young function is increasing. If there exists $s \in (0,+\infty)$ such that $\Phi(s) = +\infty$, then $\Phi(r) = +\infty$ for $r \geq s$.

Let $\mathcal{Y}$ be the set of all Young functions $\Phi$ such that
\begin{equation}\label{2.1}
0<\Phi(r)<+\infty\qquad \text{for} \qquad 0<r<+\infty
\end{equation}
If $\Phi \in \mathcal{Y}$, then $\Phi$ is absolutely continuous on every closed interval in $[0,+\infty)$
and bijective from $[0,+\infty)$ to itself.

Orlicz spaces, introduced in \cite{Orlicz1, Orlicz2}, are generalizations of Lebesgue spaces $L^p$. 
They are useful tools in harmonic analysis and its applications. For example, the Hardy-Littlewood maximal operator is bounded on $L^p$ for $1 < p < \infty$, but not on $L^1$. Using Orlicz spaces, we can investigate the boundedness of the maximal operator near $p = 1$ more precisely (see  \cite{Kita1, Kita2} and \cite{Cianchi1}).

\begin{defn} (Orlicz Space). For a Young function $\Phi$, the set
$$L^{\Phi}(\Rn)=\left\{f\in L^1_{loc}(\Rn): \int_{\Rn}\Phi(k|f(x)|)dx<+\infty
 \text{ for some $k>0$  }\right\}$$
is called Orlicz space. If $\Phi(r)=r^{p},\, 1\le p<\i$, then $L^{\Phi}(\Rn)=L^{p}(\Rn)$. If $\Phi(r)=0,\,(0\le r\le 1)$ and $\Phi(r)=\i,\,(r> 1)$, then $L^{\Phi}(\Rn)=L^{\i}(\Rn)$. The  space $L^{\Phi}_{\rm loc}(\Rn)$ endowed with the natural topology  is defined as the set of all functions $f$ such that  $f\chi_{_B}\in L^{\Phi}(\Rn)$ for all balls $B \subset \Rn$. We refer to the books \cite{KokKrbec, KrasnRut, RaoRen} for the theory of Orlicz Spaces.
\end{defn}

$L^{\Phi}(\Rn)$ is a Banach space with respect to the norm
$$\|f\|_{L^{\Phi}}=\inf\left\{\lambda>0:\int_{\Rn}\Phi\Big(\frac{|f(x)|}{\lambda}\Big)dx\leq 1\right\}.$$
We note that
$$\int_{\Rn}\Phi\Big(\frac{|f(x)|}{\|f\|_{L^{\Phi}}}\Big)dx\leq 1.$$

For a Young function $\Phi$ and  $0 \leq s \leq +\infty$, let
$$\Phi^{-1}(s)=\inf\{r\geq 0: \Phi(r)>s\}\qquad (\inf\emptyset=+\infty).$$
If $\Phi \in \mathcal{Y}$, then $\Phi^{-1}$ is the usual inverse function of $\Phi$. We note that
\begin{equation}\label{younginverse}
\Phi(\Phi^{-1}(r))\leq r \leq \Phi^{-1}(\Phi(r)) \quad \text{ for } 0\leq r<+\infty.
\end{equation}
A Young function $\Phi$ is said to satisfy the $\Delta_2$-condition, denoted by  $\Phi \in \Delta_2$, if
$$
\Phi(2r)\le k\Phi(r) \text{    for } r>0
$$
for some $k>1$. If $\Phi \in \Delta_2$, then $\Phi \in \mathcal{Y}$. A Young function $\Phi$ is said to satisfy the $\nabla_2$-condition, denoted also by  $\Phi \in \nabla_2$, if
$$\Phi(r)\leq \frac{1}{2k}\Phi(kr),\qquad r\geq 0,$$
for some $k>1$. The function $\Phi(r) = r$ satisfies the $\Delta_2$-condition but does not satisfy the $\nabla_2$-condition.
If $1 < p < \infty$, then $\Phi(r) = r^p$ satisfies both the conditions. The function $\Phi(r) = e^r - r - 1$ satisfies the
$\nabla_2$-condition but does not satisfy the $\Delta_2$-condition.

\begin{defn}
A Young function $\Phi$ is said to be of upper type p (resp. lower type p) for some $p\in[0,\i)$, if there exists a positive constant $C$ such that, for all $t\in[1,\i)$(resp. $t\in[0,1]$) and $s\in[0,\i)$,
$$
\Phi(st)\le Ct^p\Phi(s).
$$
\end{defn}

\begin{rem}\label{remlowup}
We know that if $\Phi$ is lower type $p_0$ and upper type $p_1$ with $1<p_0\le p_1<\i$, then $\Phi\in \Delta_2\cap\nabla_2$. Conversely if $\Phi\in \Delta_2\cap\nabla_2$, then  $\Phi$ is lower type $p_0$ and upper type $p_1$ with $1<p_0\le p_1<\i$ (see for example \cite{KokKrbec}).
\end{rem}

\begin{lem}\cite{Ky1}\label{Kylowupp}
Let $\Phi$ be a Young function which is lower type $p_0$ and upper type $p_1$ with $0<p_0\le p_1<\i$. Let $\widetilde{C}$ be a positive constant. Then there exists a positive constant $C$ such that
for any ball $B$ of $\Rn$ and $\mu\in(0,\i)$
$$\int_{B}\Phi\left(\frac{|f(x)|}{\mu}\right)dx\le \widetilde{C}$$
implies that $\|f\|_{L^\Phi(B)}\le C\mu$.
\end{lem}

For a Young function $\Phi$, the complementary function $\widetilde{\Phi}(r)$ is defined by
\begin{equation}\label{2.2}
\widetilde{\Phi}(r)=\left\{
\begin{array}{ccc}
\sup\{rs-\Phi(s): s\in [0,\infty)\}
& , & r\in [0,\infty) \\
+\infty&,& r=+\infty.
\end{array}
\right.
\end{equation}
The complementary function  $\widetilde{\Phi}$ is also a Young function and $\widetilde{\widetilde{\Phi}}=\Phi$. If $\Phi(r)=r$, then $\widetilde{\Phi}(r)=0$ for $0\leq r \leq 1$ and $\widetilde{\Phi}(r)=+\infty$
  for $r>1$. If $1 < p < \infty$, $1/p+1/p^\prime= 1$ and $\Phi(r) =
r^p/p$, then $\widetilde{\Phi}(r) = r^{p^\prime}/p^\prime$. If $\Phi(r) = e^r-r-1$, then $\widetilde{\Phi}(r) = (1+r) \log(1+r)-r$. Note that $\Phi \in \nabla_2$ if and only if $\widetilde{\Phi} \in \Delta_2$. It is known that
\begin{equation}\label{2.3}
r\leq \Phi^{-1}(r)\widetilde{\Phi}^{-1}(r)\leq 2r \qquad \text{for } r\geq 0.
\end{equation}

Note that Young functions satisfy the  properties
\begin{equation} \label{sam1}
\Phi(\alpha t)\leq \alpha \Phi(t)
\end{equation}
for  all $0\le\a\le1$ and  $0 \le t < \i$, and
\begin{equation} \label{sam2}
\Phi(\beta t)\geq \beta \Phi(t)
\end{equation}
for  all $\beta>1$ and  $0 \le t < \i$.

The following analogue of the H\"older inequality is known, see  \cite{Weiss}.
\begin{thm} \cite{Weiss} \label{HolderOr}
For a Young function $\Phi$ and its complementary function  $\widetilde{\Phi}$,
the following inequality is valid
$$\|fg\|_{L^{1}(\Rn)} \leq 2 \|f\|_{L^{\Phi}} \|g\|_{L^{\widetilde{\Phi}}}.$$
\end{thm}

The following lemma is valid.
\begin{lem}\label{lem4.0}  \cite{BenSharp, LiuWang, Nakai0}
Let $\Phi$ be a Young function and $B$ a set in $\mathbb{R}^n$ with finite Lebesgue measure. Then
$$
\|\chi_{_B}\|_{WL_{\Phi}(\Rn)} = \|\chi_{_B}\|_{L^{\Phi}(\Rn)} = \frac{1}{\Phi^{-1}\left(|B|^{-1}\right)}.
$$
\end{lem}

In the next sections where we prove our main estimates, we use the following lemma, which follows  from Theorem \ref{HolderOr}, Lemma \ref{lem4.0} and \eqref{2.3}.
\begin{lem}\label{lemHold}
For a Young function $\Phi$ and $B=B(x,r)$, the following inequality is valid
$$\|f\|_{L^{1}(B)} \leq 2 |B| \Phi^{-1}\left(|B|^{-1}\right) \|f\|_{L^{\Phi}(B)} .$$
\end{lem}

\begin{defn} (generalized Orlicz-Morrey Space)
Let $\varphi(x,r)$ be a positive measurable function on $\Rn \times (0,\infty)$ and $\Phi$ any Young function.
We denote by $M^{\Phi,\varphi}(\Rn)$ the generalized Orlicz-Morrey space, the space of all
functions $f\in L^{\Phi}_{\rm loc}(\Rn)$ with finite quasinorm
$$
\|f\|_{M^{\Phi,\varphi}} = \sup\limits_{x\in\Rn, r>0}
\varphi(x,r)^{-1} \Phi^{-1}(r^{-n}) \|f\|_{L^{\Phi}(B(x,r))}.
$$
\end{defn}
According to this definition, we recover the generalized Morrey space $M^{p,\varphi}$ under the choice $\Phi(r)=r^{p},\,1\le p<\i$. If $\Phi(r)=r^{p},\,1\le p<\i$ and $\varphi(x,r)=r^{\frac{\lambda-n}{p}},\,0\le \lambda \le n$, then $M^{\Phi,\varphi}(\Rn)$ coincides with the Morrey space $M^{p,\lambda}(\Rn)$ and if $\varphi(x,r)=\Phi^{-1}(r^{-n})$, then $M^{\Phi,\varphi}(\Rn)$ coincides with the Orlicz space $L^{\Phi}(\Rn)$.

\

\section{Intrinsic square functions in the spaces $M^{\Phi,\varphi}(\Rn)$}

The known boundedness statement for $G_{\alpha}$ on Orlicz spaces runs as follows.
\begin{thm}\cite{LiNaYaZh}\label{LiNaYaZhprop2.11-1}
Let $\a\in(0,1]$, $\Phi$ be a Young function which is lower type $p_0$ and upper type $p_1$ with $1<p_0\le p_1<\i$. Then $G_{\alpha}$ is bounded from $L^{\Phi}(\mathbb{R}^{n})$ to itself.
\end{thm}

We will  use the following statement on the boundedness
of the weighted Hardy operator
$$
H^{\ast}_{w} g(t):=\int_t^{\infty} g(s) w(s) ds,~ \ \  0<t<\infty,
$$
where $w$ is a weight.

The following theorem was proved in \cite{GulJMS2013}. 

\begin{thm}\label{thm3.2.}
Let $v_1$, $v_2$ and $w$ be weights on  $(0,\infty)$ and $v_1(t)$ be bounded outside a neighborhood of the
origin. The inequality
\begin{equation} \label{vav01}
\sup _{t>0} v_2(t) H^{\ast}_{w} g(t) \leq C \sup _{t>0} v_1(t) g(t)
\end{equation}
holds for some $C>0$ for all non-negative and non-decreasing $g$ on $(0,\i)$ if and
only if
\begin{equation} \label{vav02}
B:= \sup _{t>0} v_2(t)\int_t^{\infty} \frac{w(s) ds}{\sup _{s<\tau<\infty} v_1(\tau)}<\infty.
\end{equation}
Moreover, the value  $C=B$ is the best constant for  \eqref{vav01}.
\end{thm}

\begin{rem}\label{rem2.3.}
In \eqref{vav01} and \eqref{vav02} it is assumed that $\frac{1}{\i}=0$ and $0 \cdot \i=0$.
\end{rem}

The following lemma was generalization of the Guliyev lemma \cite{GulDoc, GulBook, GulJIA} for Orlicz spaces.

\begin{lem}\label{lem3.3.} Let $\a\in(0,1]$, $\Phi$ be a Young function which is lower type $p_0$ and upper type $p_1$ with $1<p_0\le p_1<\i$, $f\in L^{\Phi}_{\rm loc}(\Rn)$, $B=B(x_0,r)$, $x_0\in \mathbb{R}^n$ and $r>0$. Then for the operator $G_{\alpha}$ the following inequality is valid
\begin{equation}\label{CZdgs}
\|G_{\alpha} f\|_{L^{\Phi}(B)} \lesssim \frac{1}{\Phi^{-1}\big(r^{-n}\big)}
 \int_{2r}^{\i} \|f\|_{L^{\Phi}(B(x_0,t))}  \Phi^{-1}\big(t^{-n}\big) \frac{dt}{t}.
\end{equation}
\end{lem}

\begin{proof}
With the notation $2B=B(x_0,2r)$, we represent  $f$ as
$$
f=f_1+f_2, \ \quad f_1(y)=f(y)\chi _{2B}(y),\quad
 f_2(y)=f(y)\chi_{\dual {(2B)}}(y),
$$
and then
$$
\|G_{\alpha}f\|_{L^{\Phi}(B)} \le \|G_{\alpha}f_1\|_{L^{\Phi}(B)} +\|G_{\alpha}f_2\|_{L^{\Phi}(B)}.
$$
Since $f_1\in L^{\Phi}(\Rn)$, by Theorem \ref{LiNaYaZhprop2.11-1}, it follows that
\begin{equation*}
\|G_{\alpha}f_1\|_{L^\Phi(B)}\leq \|G_{\alpha}f_1\|_{L^\Phi(\Rn)}\leq C\|f_1\|_{L^\Phi(\Rn)}=C\|f\|_{L^\Phi(2B)}.
\end{equation*}
Then let us estimate $\|G_{\alpha}f_2\|_{L^{\Phi}(B)}$.
$$
|f_{2}*\displaystyle \phi_{t}(y)|=\left|t^{-n}\int_{|y-z|\leq t}\phi(\frac{y-z}{t})f_{2}(z)dz\right|\leq t^{-n}\int_{|y-z|\leq t}|f_{2}(z) |dz.
$$
Since $x\in B(x_{0},r)$, $(y,t)\in\Gamma(x)$, we have $|x-z|\leq|z-y|+|x-y|\leq 2t$, and
$$
r\leq|x_{0}-z|-|x_{0}-x|\leq|x-z|\leq|x-y|+|y-z|\leq 2t.
$$
So, we obtain
\begin{eqnarray*}
G_{\alpha}f_{2}(x) &\leq& \left(\int \int_{\Gamma(x)}\left|t^{-n}\int_{|y-z|\leq t}|f_{2}(z)|dz\right|^{2}\frac{dydt}{t^{n+1}}\right)^{\frac{1}{2}}
\\
{}&\leq&\left(\int_{t>r/2}\int_{|x-y|<t}\left(\int_{|x-z|\leq 2t}|f_{2}(z)|dz\right)^{2}\frac{dydt}{t^{3n+1}}\right)^{\frac{1}{2}}
\\
{}&\lesssim& \left(\int_{t>r/2}\left(\int_{|x-z|\leq 2t}|f_{2}(z)|dz\right)^{2}\frac{dt}{t^{2n+1}}\right)^{\frac{1}{2}}.
\end{eqnarray*}
By Minkowski's inequality and $|x-z|\displaystyle \geq|x_{0}-z|-|x_{0}-x|\geq\frac{1}{2}|x_{0}-z|$, we have
\begin{eqnarray}\label{estGf2}
G_{\alpha}f_{2}(x) &\lesssim& \int_{\mathbb{R}^{n}}\left(\int_{t>\frac{|x-z|}{2}}\frac{dt}{t^{2n+1}}\right)^{\frac{1}{2}}|f_{2}(z)|dz \nonumber
\\
{}&\lesssim&\int_{|x_{0}-z|>2r}\frac{|f(z)|}{|x-z|^{n}}dz\lesssim\int_{|x_{0}-z|>2r}\frac{|f(z)|}{|x_{0}-z|^{n}}dz \nonumber
\\
{}&=&\int_{|x_{0}-z|>2r}|f(z)|\int_{|x_{0}-z|}^{+\infty}\frac{dt}{t^{n+1}}dz \nonumber
\\
{}&=&\int_{2r}^{\infty}\int_{2r<|x_{0}-z|<t}|f(z)|dz\, \frac{dt}{t^{n+1}} \nonumber
\\
{}&\lesssim&\int_{2r}^{\i} \|f\|_{L^{\Phi}(B(x_0,t))}  \Phi^{-1}\big(t^{-n}\big) \frac{dt}{t}.
\end{eqnarray}
The last inequality is by Lemma \ref{lemHold}. Moreover,
\begin{equation} \label{ves2}
\|G_{\alpha}f_2\|_{L^{\Phi}(B)}\lesssim
\frac{1}{\Phi^{-1}\big(r^{-n}\big)}\int_{2r}^{\i}\|f\|_{L^{\Phi}(B(x_0,t))}
\Phi^{-1}\big(t^{-n}\big) \frac{dt}{t}.
\end{equation}
Thus
\begin{equation*}
\|G_{\alpha}f\|_{L^\Phi(B)}\lesssim \|f\|_{L^{\Phi}(2B)}+ \frac{1}{\Phi^{-1}\big(r^{-n}\big)}
\int_{2r}^{\i}\|f\|_{L^\Phi (B(x_0,t))}\Phi^{-1}\big(t^{-n}\big) \frac{dt}{t}.
\end{equation*}
On the other hand, by  \eqref{2.3} we get
\begin{eqnarray*}
\Phi^{-1}\big(r^{-n}\big)&\thickapprox & \Phi^{-1}\big(r^{-n}\big) r^n \int_{2r}^{\i}\frac{dt}{t^{n+1}}
\\
& \lesssim & \int_{2r}^{\i} \Phi^{-1}\big(t^{-n}\big) \frac{dt}{t}
\end{eqnarray*}
and then
\begin{equation} \label{ves2fd}
\|f\|_{L^{\Phi}(2B)}\lesssim
\frac{1}{\Phi^{-1}\big(r^{-n}\big)}\int_{2r}^{\i} \|f\|_{L^{\Phi}(B(x_0,t))} \Phi^{-1}\big(t^{-n}\big) \frac{dt}{t}.
\end{equation}
Thus
\begin{equation*}
\|G_{\alpha}f\|_{L^{\Phi}(B)}\lesssim \frac{1}{\Phi^{-1}\big(r^{-n}\big)} \int_{2r}^{\i}
\|f\|_{L^\Phi (B(x_0,t))} \Phi^{-1}\big(t^{-n}\big) \frac{dt}{t}.
\end{equation*}
\end{proof}

\begin{thm}\label{thm4.4.}
Let $\a\in(0,1]$, $\Phi$ be a Young function which is lower type $p_0$ and upper type $p_1$ with $1<p_0\le p_1<\i$ and the functions $(\varphi_1,\varphi_2)$ and $\Phi$ satisfy the condition
\begin{equation}\label{eq3.6.VZPot}
\int_{r}^{\i} \Big(\es_{t<s<\i}\frac{\varphi_1(x,s)}{\Phi^{-1}\big(s^{-n}\big)}\Big) \, \Phi^{-1}\big(t^{-n}\big)\frac{dt}{t}  \le C \, \varphi_2(x,r),
\end{equation}
where $C$ does not depend on $x$ and $r$.
Then $G_{\alpha}$ is bounded from $M^{\Phi,\varphi_1}(\Rn)$ to $M^{\Phi,\varphi_2}(\Rn)$.
\end{thm}
\begin{proof}
By Lemma \ref{lem3.3.} and Theorem \ref{thm3.2.} we have
\begin{equation*}
\begin{split}
\|G_{\a} f\|_{M^{\Phi,\varphi_2}(\Rn)} & \lesssim \sup\limits_{x\in\Rn, r>0}
\varphi_2(x,r)^{-1}\,
\int_{r}^{\i} \Phi^{-1}\big(t^{-n}\big) \|f\|_{L^{\Phi}(B(x,t))}\frac{dt}{t}
\\
& \lesssim \sup\limits_{x\in\Rn, r>0}
\varphi_1(x,r)^{-1} \Phi^{-1}\big(r^{-n}\big) \|f\|_{L^{\Phi}(B(x,r))}
\\
& = \|f\|_{M^{\Phi,\varphi_1}(\Rn)}.
\end{split}
\end{equation*}
\end{proof}

In \cite{Wilson1}, the author proved that the functions $G_{\alpha}f$ and $g_{\alpha}f$ are pointwise comparable. Thus, as a consequence of Theorem \ref{thm4.4.}, we have the following result.
\begin{cor}\label{wucor1.5}
Let $\a\in(0,1]$, $\Phi$ be a Young function which is lower type $p_0$ and upper type $p_1$ with $1<p_0\le p_1<\i$ and the functions $(\varphi_1,\varphi_2)$ and $\Phi$ satisfy the condition \eqref{eq3.6.VZPot}, then $g_{\alpha}$ is bounded from $M^{\Phi,\varphi_1}(\Rn)$ to $M^{\Phi,\varphi_2}(\Rn)$.
\end{cor}

The following lemma is an easy consequence of the inequality $$G_{\alpha,\beta}(f)(x)\leq\beta^{\frac{3n}{2}+\alpha}G_{\alpha}(f)(x)$$ which was proved in \cite{Wilson1} and the monotonicity of the norm $\|\cdot\|_{L^{\Phi}}$.
\begin{lem}\label{wulem2.3}
For $j\in \mathrm{Z}^{+}$, denote
$$
G_{\alpha,2^{j}}(f)(x)=\left(\int_{0}^{\infty}\int_{|x-y|\leq 2^{j}t}(A_{\alpha}f(y,t))^{2}\frac{dydt}{t^{n+1}}\right)^{\frac{1}{2}}
$$
Let $\Phi$ be a Young function and $0<\alpha\leq 1$, then we have
$$
\Vert G_{\alpha,2^{j}}(f)\Vert_{L^{\Phi}(\mathbb{R}^{n})}\lesssim 2^{j(\frac{3n}{2}+\alpha)}\Vert G_{\alpha}(f)\Vert_{L^{\Phi}(\mathbb{R}^{n})}.
$$
\end{lem}

\begin{thm}\label{thm4.4.x}
Let $\a\in(0,1]$, $\Phi$ be a Young function which is lower type $p_0$ and upper type $p_1$ with $1<p_0\le p_1<\i$ and the functions $(\varphi_1,\varphi_2)$ and $\Phi$ satisfy the condition
\begin{equation}\label{eq3.6.VZPotx}
\int_{r}^{\i} \es_{t<s<\i}\frac{\varphi_1(x,s)}{\Phi^{-1}\big(s^{-n}\big)}\Phi^{-1}\big(t^{-n}\big)\frac{dt}{t}  \le C \, \varphi_2(x,r),
\end{equation}
where $C$ does not depend on $x$ and $r$.
Then for $\lambda >3+\displaystyle \frac{2\alpha}{n}$, $g_{\lambda,\alpha}^{*}$ is bounded from $M^{\Phi,\varphi_1}(\Rn)$ to $M^{\Phi,\varphi_2}(\Rn)$.
\end{thm}
\begin{proof}
\begin{eqnarray*}
[g_{\lambda,\alpha}^{*}(f)(x)]^{2}& = &\displaystyle \int_{0}^{\infty}\int_{|x-y|<t}\left(\frac{t}{t+|x-y|}\right)^{n\lambda}(A_{\alpha}f(y,t))^{2}\frac{dydt}{t^{n+1}}
\\
{}&{}&+\int_{0}^{\infty}\int_{|x-y|\geq t}\left(\frac{t}{t+|x-y|}\right)^{n\lambda}(A_{\alpha}f(y,t))^{2}\frac{dydt}{t^{n+1}}
\\
{}&:=&I+II.
\end{eqnarray*}
First, let us estimate I.
$$
I\leq \displaystyle \int_{0}^{+\infty}\int_{|x-y|<t}(A_{\alpha}f(y,t))^{2}\frac{dydt}{t^{n+1}}\leq(G_{\alpha}f(x))^{2}
$$
Now, let us estimate II.
\begin{eqnarray*}
II &\leq& \displaystyle \sum_{j=1}^{\infty}\int_{0}^{\infty}\int_{2^{j-1}t\leq|x-y|\leq 2^{j}t}\left(\frac{t}{t+|x-y|}\right)^{n\lambda}(A_{\alpha}f(y,t))^{2}\frac{dydt}{t^{n+1}}
\\
{}&\lesssim &\sum_{j=1}^{\infty}\int_{0}^{\infty}\int_{2^{j-1}t\leq|x-y|\leq 2^{j}t}2^{-jn\lambda}(A_{\alpha}f(y,t))^{2}\frac{dydt}{t^{n+1}}
\\
{}&\lesssim & \sum_{j=1}^{\infty}2^{-jn\lambda}\int_{0}^{\infty}\int_{|x-y|\leq 2^{j}t}(A_{\alpha}f(y,t))^{2}\frac{dydt}{t^{n+1}}
\\
{}&:=&\sum_{j=1}^{\infty}2^{-jn\lambda}(G_{\alpha,2^{j}}(f)(x))^{2}
\end{eqnarray*}
Thus,
\begin{equation}\label{wueq7}
\Vert g_{\lambda,\alpha}^{*}(f) \Vert_{M^{\Phi,\varphi_2}(\Rn)}\leq\Vert G_{\alpha}f \displaystyle \Vert_{M^{\Phi,\varphi_2}(\Rn)}+\sum_{j=1}^{\i} 2^{-\frac{jn\lambda}{2}}\Vert G_{\alpha,2^{j}}(f) \Vert_{M^{\Phi,\varphi_2}(\Rn)}.
\end{equation}
By Theorem \ref{thm4.4.}, we have
\begin{equation}\label{wueq8}
\Vert G_{\alpha}f\Vert_{M^{\Phi,\varphi_2}(\Rn)}\lesssim\Vert f\Vert_{M^{\Phi,\varphi_1}(\Rn)}.
\end{equation}
In the following, we will estimate $\Vert G_{\alpha,2^{j}}(f) \Vert_{M^{\Phi,\varphi_2}(\Rn)}$. We divide $\Vert G_{\alpha,2^{j}}(f)\Vert_{L^{\Phi}(B)}$ into two parts.
\begin{equation}\label{wueq9}
\Vert G_{\alpha,2^{j}}(f)\Vert_{L^{\Phi}(B)}\leq\Vert G_{\alpha,2^{j}}(f_{1})\Vert_{L^{\Phi}(B)}+\Vert G_{\alpha,2^{j}}(f_{2})\Vert_{L^{\Phi}(B)},
\end{equation}
where $f_{1}(y)=f(y)\chi_{2B}(y)$ , $f_{2}(y)=f(y)-f_{1}(y)$ . For the first part, by Lemma \ref{wulem2.3} and \eqref{ves2fd},
$$
\Vert G_{\alpha,2^{j}}(f_{1})\Vert_{L^{\Phi}(B)} \lesssim 2^{j(\frac{3n}{2}+\alpha)}\Vert G_{\alpha}(f_{1})\Vert_{L^{\Phi}(\mathbb{R}^{n})}\lesssim2^{j(\frac{3n}{2}+\alpha)}\Vert f\Vert_{L^{\Phi}(2B)}
$$
\begin{equation}\label{wueq10}
\lesssim 2^{j(\frac{3n}{2}+\alpha)}\frac{1}{\Phi^{-1}\big(r^{-n}\big)}\int_{2r}^{\i} \|f\|_{L^{\Phi}(B(x_0,t))} \Phi^{-1}\big(t^{-n}\big) \frac{dt}{t}.
\end{equation}
For the second part.
\begin{eqnarray*}
G_{\alpha,2^{j}}(f_{2})(x) &=& \left(\displaystyle \int_{0}^{\infty}\int_{|x-y|\leq 2^{j}t}(A_{\alpha}f(y,t))^{2}\frac{dydt}{t^{n+1}}\right)^{\frac{1}{2}}
\\
{}&=& \left(\int_{0}^{\infty}\int_{|x-y|\leq 2^{j}t}\left(\sup_{\phi\in C_{\alpha}}|f*\phi_{t}(y)|\right)^{2}\frac{dydt}{t^{n+1}}\right)^{\frac{1}{2}}
\\
{}&\leq& \left(\int_{0}^{\infty}\int_{|x-y|\leq 2^{j}t}\left(\int_{|z-y|\leq t}|f_{2}(z)|dz\right)^{2}\frac{dydt}{t^{3n+1}}\right)^{\frac{1}{2}}.
\end{eqnarray*}
Since $|x-z|\leq|z-y|+|x-y|\leq 2^{j+1}t$, we get
\begin{eqnarray*}
G_{\alpha,2^{j}}(f_{2})(x) &\leq& \left(\displaystyle \int_{0}^{\infty}\int_{|x-y|\leq 2^{j}t}\left(\int_{|x-z|\leq 2^{j+1}t}|f_{2}(z)|dz\right)^{2}\frac{dydt}{t^{3n+1}}\right)^{\frac{1}{2}}
\\
{}&\leq& \left(\int_{0}^{\infty}\left(\int_{|x-z|\leq 2^{j+1}t}|f_{2}(z)|dz\right)^{2}\frac{2^{jn}dt}{t^{2n+1}}\right)^{\frac{1}{2}}
\\
{}&\leq& 2^{\frac{jn}{2}}\int_{\mathbb{R}^{n}}\left(\int_{t\geq\frac{|x-z|}{2^{j+1}}}|f_{2}(z)|^{2}\frac{1}{t^{2n+1}}dt\right)^{\frac{1}{2}}dz
\\
{}&\leq& 2^{\frac{3jn}{2}}\int_{|x_{0}-z|>2r}\frac{|f(z)|}{|x-z|^{n}}dz.
\end{eqnarray*}
For $|x-z|\displaystyle \geq|x_{0}-z|-|x_{0}-x|\geq|x_{0}-z|-\frac{1}{2}|x_{0}-z|=\frac{1}{2}|x_{0}-z|$, so by Fubini's theorem and Lemma \ref{lemHold}, we obtain
\begin{eqnarray*}
G_{\alpha,2^{j}}(f_{2})(x)&\leq& 2^{\frac{3jn}{2}}\int_{|x_{0}-z|>2r}\frac{|f(z)|}{|x_{0}-z|^{n}}dz
\\
{}&=& 2^{\frac{3jn}{2}}\int_{|x_{0}-z|>2r}|f(z)|\int_{|x_{0}-z|}^{\infty}\frac{1}{t^{n+1}}dtdz
\\
{}&\leq& 2^{\frac{3jn}{2}}\int_{2r}^{\infty}\int_{|x_{0}-z|<t}|f(z)|\frac{1}{t^{n+1}}dzdt
\\
{}&\leq& 2^{\frac{3jn}{2}}\int_{2r}^{\i} \|f\|_{L^{\Phi}(B(x_0,t))}  \Phi^{-1}\big(t^{-n}\big) \frac{dt}{t}.
\end{eqnarray*}
So,
\begin{equation}\label{wueq11}
\Vert G_{\alpha,2^{j}}(f_{2})\Vert_{L^{\Phi}(B)}\lesssim 2^{\frac{3jn}{2}}\frac{1}{\Phi^{-1}\big(r^{-n}\big)}\int_{2r}^{\i} \|f\|_{L^{\Phi}(B(x_0,t))} \Phi^{-1}\big(t^{-n}\big) \frac{dt}{t}.
\end{equation}
Combining \eqref{wueq9}, \eqref{wueq10} and \eqref{wueq11}, we have
$$
\Vert G_{\alpha,2^{j}}(f)\Vert_{L^{\Phi}(B)}\lesssim 2^{j(\frac{3n}{2}+\alpha)}\frac{1}{\Phi^{-1}\big(r^{-n}\big)}\int_{2r}^{\i} \|f\|_{L^{\Phi}(B(x_0,t))} \Phi^{-1}\big(t^{-n}\big) \frac{dt}{t}.
$$

Thus by Theorem \ref{thm3.2.} we have
\begin{eqnarray}\label{wueq12}
\|G_{\alpha,2^{j}} f\|_{M^{\Phi,\varphi_2}(\Rn)} & \lesssim & 2^{j(\frac{3n}{2}+\alpha)}\sup\limits_{x\in\Rn, r>0}
\varphi_2(x,r)^{-1}\,
\int_{r}^{\i} \Phi^{-1}\big(t^{-n}\big) \|f\|_{L^{\Phi}(B(x,t))}\frac{dt}{t}\nonumber
\\
{}& \lesssim & 2^{j(\frac{3n}{2}+\alpha)}\sup\limits_{x\in\Rn, r>0}
\varphi_1(x,r)^{-1} \Phi^{-1}\big(r^{-n}\big) \|f\|_{L^{\Phi}(B(x,r))}\nonumber
\\
{}& = & 2^{j(\frac{3n}{2}+\alpha)}\|f\|_{M^{\Phi,\varphi_1}(\Rn)}.
\end{eqnarray}
Since $\lambda >3+\displaystyle \frac{2\alpha}{n}$, by \eqref{wueq7}, \eqref{wueq8} and \eqref{wueq12}, we have the desired theorem.
\end{proof}

\

\section{Commutators of the intrinsic square functions in the spaces $M^{\Phi,\varphi}(\Rn)$}

It is well known that the commutator is an important integral operator and plays a key role in harmonic analysis. In 1965, Calderon \cite{Cald1, Cald2} studied a kind of
commutators appearing in Cauchy integral problems of Lip-line. Let $K$ be a Calder\'{o}n--Zygmund singular integral operator and $b \in BMO(\Rn)$.
A well known result of Coifman, Rochberg and Weiss \cite{CRW} states that the commutator operator $[b, K] f = K(b f)-b \, Kf$ is bounded on $L_p(\Rn)$ for $1 < p < \infty$.
The commutator of Calder\'{o}n--Zygmund operators plays an important role in studying the regularity of solutions of elliptic partial differential equations of second order. The boundedness result was generalized to other contexts and important applications to some non-linear PDEs were given by Coifman et al. \cite{CLMS}.

A function $f\in L^{1}_{loc}(\mathbb{R}^{n})$ is said to be in $BMO(\Rn)$ if
$$
\Vert f\Vert_{*}=\sup_{x\in \mathbb{R}^{n},r>0}\frac{1}{|B(x,r)|}\int_{B(x,r)}|f(y)-f_{B(x,r)}|dy<\infty,
$$
where $f_{B(x,r)}=\displaystyle \frac{1}{|B(x,r)|}\int_{B(x,r)}f(y)dy$.

Before proving the main theorems, we need the following lemmas.
\begin{lem} \label{rem2.4.}\cite{JohnNirenberg}
$(1)~~$ The John--Nirenberg inequality:
there are constants $C_1$, $C_2>0$, such that for all $b \in BMO(\Rn)$ and $\beta>0$
$$
\left| \left\{ x \in B \, : \, |b(x)-b_{B}|>\beta \right\}\right|
\le C_1 |B| e^{-C_2 \beta/\| b \|_{\ast}}, ~~~ \forall B \subset \Rn.
$$

$(2)~~$ The John--Nirenberg inequality implies that
\begin{equation} \label{lem2.4.xx}
\|b\|_\ast \thickapprox \sup_{x\in\Rn, r>0}\left(\frac{1}{|B(x,r)|} \int_{B(x,r)}|b(y)-b_{B(x,r)}|^p dy\right)^{\frac{1}{p}}
\end{equation}
for $1<p<\infty$.

$(3)~~$ Let $b \in BMO(\Rn)$. Then there is a constant $C>0$ such that
\begin{equation} \label{propBMO}
\left|b_{B(x,r)}-b_{B(x,t)}\right| \le C \|b\|_\ast \ln \frac{t}{r} \;\;\; \mbox{for} \;\;\; 0<2r<t,
\end{equation}
where $C$ is independent of $b$, $x$, $r$ and $t$.
\end{lem}

\begin{lem}\label{Bmo-orlicz}
Let $b\in BMO$ and $\Phi$ be a Young function. Let $\Phi$ is lower type $p_0$ and upper type $p_1$ with $1<p_0\le p_1<\i$, then
$$
\|b\|_\ast \thickapprox \sup_{x\in\Rn, r>0}\Phi^{-1}\big(r^{-n}\big)\left\|b(\cdot)-b_{B(x,r)}\right\|_{L^{\Phi}(B(x,r))}.
$$
\end{lem}

\begin{proof}
By H\"{o}lder's inequality, we have
$$\|b\|_\ast \lesssim \sup_{x\in\Rn, r>0}\Phi^{-1}\big(r^{-n}\big)\left\|b(\cdot)-b_{B(x,r)}\right\|_{L^{\Phi}(B(x,r))}.$$

Now we show that
$$
\sup_{x\in\Rn, r>0}\Phi^{-1}\big(r^{-n}\big)\left\|b(\cdot)-b_{B(x,r)}\right\|_{L^{\Phi}(B(x,r))} \lesssim \|b\|_\ast .
$$
Without loss of generality, we may assume that $\|b\|_\ast=1$; otherwise,
we replace $b$ by $b/\|b\|_\ast$. By the fact that $\Phi$ is lower type $p_0$ and upper type $p_1$ and \eqref{younginverse} it follows that
$$
\int_{B(x,r)}\Phi\left(\frac{|b(y)-b_{B(x,r)}|\,\Phi^{-1}\big(|B(x,r)|^{-1}\big)}{\|b\|_\ast}\right)dy
$$
$$
=\int_{B(x,r)}\Phi\left(|b(y)-b_{B(x,r)}|\Phi^{-1}\big(|B(x,r)|^{-1}\big)\right)dy
$$
$$
\lesssim\frac{1}{|B(x,r)|}\int_{B(x,r)}\left[|b(y)-b_{B(x,r)}|^{p_0}+|b(y)-b_{B(x,r)}|^{p_1}\right]dy\lesssim 1.
$$
By Lemma \ref{Kylowupp} we get the desired result.
\end{proof}
\begin{rem}
Note the, the Lemma \ref{Kylowupp} for the variable exponent Lebesgue space $L^{p(\cdot)}$ case were proved in \cite{IzukiSaw}.
\end{rem}

The known boundedness statement for $[b,G_{\alpha}]$ on Orlicz spaces runs as follows.
\begin{thm}\cite{LiNaYaZh}\label{LiNaYaZhprop2.17-1}
Let $b\in BMO$, $\a\in(0,1]$, $\Phi$ be a Young function which is lower type $p_0$ and upper type $p_1$ with $1<p_0\le p_1<\i$. Then $[b,G_{\alpha}]$ is bounded from $L^{\Phi}(\mathbb{R}^{n})$ to itself.
\end{thm}

\begin{lem}\label{lem5.1.}
Let $b\in BMO$, $\a\in(0,1]$, $\Phi$ be a Young function which is lower type $p_0$ and upper type $p_1$ with $1<p_0\le p_1<\i$. Then the inequality
\begin{equation*}\label{eq5.1.}
\|[b,G_{\alpha}]f\|_{L^\Phi(B(x_0,r))} \lesssim  \frac{\|b\|_{*}}{\Phi^{-1}\big(r^{-n}\big)}
 \int_{2r}^{\i} \Big(1+\ln \frac{t}{r}\Big)\|f\|_{L^{\Phi}(B(x_0,t))}  \Phi^{-1}\big(t^{-n}\big) \frac{dt}{t}
\end{equation*}
holds for any ball $B(x_0,r)$ and for all $f\in L^{\Phi}_{\rm loc}(\Rn)$.
\end{lem}
\begin{proof}
For arbitrary $x_0 \in\Rn$, set $B=B(x_0,r)$ for the ball centered at $x_0$ and of radius $r$. Write $f=f_1+f_2$ with
$f_1=f\chi_{_{2B}}$ and $f_2=f\chi_{_{\dual (2B)}}$. Hence
$$
\left\|[b,G_{\alpha}]f \right\|_{L^\Phi(B)} \leq
\left\|[b,G_{\alpha}]f_1 \right\|_{L^\Phi(B)}+
\left\|[b,G_{\alpha}]f_2 \right\|_{L^\Phi(B)}.
$$
From the boundedness of $[b,G_{\alpha}]$ in $L^\Phi(\Rn)$  
it follows that
\begin{align*}
\|[b,G_{\alpha}]f_1\|_{L^\Phi(B)} & \leq
\|[b,G_{\alpha}]f_1\|_{L^\Phi(\Rn)}
\\
& \lesssim  \|b\|_{*} \, \|f_1\|_{L^\Phi(\Rn)} = \|b\|_{*} \, \|f\|_{L^\Phi(2B)}.
\end{align*}
For the second part, we divide it into two parts.
\begin{eqnarray*}
[b,G_{\alpha}]f_{2}(x) &=& \left(\displaystyle \int\int_{\Gamma(x)}\sup_{\phi\in C_{\alpha}}\left|\int_{\mathbb{R}^{n}}[b(x)-b(z)]\phi_{t}(y-z)f_{2}(z)dz\right|^{2}\frac{dydt}{t^{n+1}}\right)^{\frac{1}{2}}
\\
{}&\leq&\left(\int\int_{\Gamma(x)}\sup_{\phi\in C_{\alpha}}\left|\int_{\mathbb{R}^{n}}[b(x)-b_{B}]\phi_{t}(y-z)f_{2}(z)dz\right|^{2}\frac{dydt}{t^{n+1}}\right)^{\frac{1}{2}}
\\
{}&{}&+\left(\int\int_{\Gamma(x)}\sup_{\phi\in C_{\alpha}}\left|\int_{\mathbb{R}^{n}}[b_{B}-b(z)]\phi_{t}(y-z)f_{2}(z)dz\right|^{2}\frac{dydt}{t^{n+1}}\right)^{\frac{1}{2}}
\\
{}&:=& A+B.
\end{eqnarray*}
First, for $A$, we find that
$$
A=|b(x)-b_{B}|\left(\iint_{\Gamma(x)}\sup_{\phi\in C_{\alpha}}\left|\int_{\mathbb{R}^{n}}\phi_{t}(y-z)f_{2}(z)dz\right|^{2}\frac{dydt}{t^{n+1}}\right)^{\frac{1}{2}}=|b(x)-b_{B}|G_{\alpha}f_{2}(x).
$$
By \eqref{estGf2} and Lemma \ref{Bmo-orlicz}
, we can get
\begin{eqnarray*}
\left\|A\right\|_{L^\Phi(B)}&=&\left\||b(\cdot)-b_{B}|G_{\alpha}f_{2}(\cdot)\right\|_{L^\Phi(B)}
\\
{}&\leq&\left\|b(\cdot)-b_{B}\right\|_{L^\Phi(B)}\int_{2r}^{\i} \|f\|_{L^{\Phi}(B(x_0,t))}  \Phi^{-1}\big(t^{-n}\big) \frac{dt}{t}
\\
{}&\leq&\frac{\|b\|_{*}}{\Phi^{-1}\big(r^{-n}\big)}\int_{2r}^{\i} \|f\|_{L^{\Phi}(B(x_0,t))}  \Phi^{-1}\big(t^{-n}\big) \frac{dt}{t}.
\end{eqnarray*}
For $B$, since $|x-y|<t$, we get $|x-z|<2t$. Thus, by Minkowski's inequality,
\begin{eqnarray*}
B &\leq& \left(\displaystyle \int\int_{\Gamma(x)}\left|\int_{|x-z|<2t}|b_{B}-b(z)||f_{2}(z)|dz\right|^{2}\frac{dydt}{t^{3n+1}}\right)^{\frac{1}{2}}
\\
{}&\lesssim& \left(\displaystyle \int_{0}^{\infty}\left|\int_{|x-z|<2t}|b_{B}-b(z)||f_{2}(z)|dz\right|^{2}\frac{dt}{t^{2n+1}}\right)^{\frac{1}{2}}
\\
{}& \leq &  \displaystyle \int_{|x_{0}-z|>2r}|b_{B}-b(z)||f(z)|\frac{dz}{|x-z|^{n}}.
\\
\end{eqnarray*}
Since $|x-z|\displaystyle \geq\frac{1}{2}|x_{0}-z|$, we have
\begin{align*}
\left\|B\right\|_{L^\Phi(B)}&\lesssim \Big\|\int_{\dual (2B)} \frac{|b(z)-b_{B}|}{|x_0-z|^{n}}|f(z)|dz\Big\|_{L^\Phi(B)}
\\
&\thickapprox \frac{1}{\Phi^{-1}\big(r^{-n}\big)}\int_{\dual
(2B)}\frac{|b(z)-b_{B}|}{|x_0-z|^{n}}|f(z)|dz
\\
&\thickapprox \frac{1}{\Phi^{-1}\big(r^{-n}\big)}\int_{\dual
(2B)}|b(z)-b_{B}||f(z)|\int_{|x_0-z|}^{\infty}\frac{dt}{t^{n+1}}dz
\\
&\thickapprox \frac{1}{\Phi^{-1}\big(r^{-n}\big)} \int_{2r}^{\infty}\int_{2r\leq |x_0-z|\leq t}
|b(z)-b_{B}||f(z)|dz\frac{dt}{t^{n+1}}
\\
&\lesssim \frac{1}{\Phi^{-1}\big(r^{-n}\big)} \int_{2r}^{\infty}\int_{B(x_0,t)}
|b(z)-b_{B}||f(z)|dz\frac{dt}{t^{n+1}}.
\end{align*}

Applying H\"older's inequality, by Lemma \ref{Bmo-orlicz} and \eqref{propBMO}  we get
\allowdisplaybreaks
\begin{align*}
\left\|B\right\|_{L^\Phi(B)} & \lesssim \frac{1}{\Phi^{-1}\big(r^{-n}\big)} \int_{2r}^{\infty}\int_{B(x_0,t)}
|b(z)-b_{B(x_0,t)}||f(z)|dz\frac{dt}{t^{n+1}}
\\
&\quad + \frac{1}{\Phi^{-1}\big(r^{-n}\big)} \int_{2r}^{\infty}|b_{B}-b_{B(x_0,t)}|
\int_{B(x_0,t)} |f(z)|dz\frac{dt}{t^{n+1}}
\\
&\lesssim \frac{1}{\Phi^{-1}\big(r^{-n}\big)} \int_{2r}^{\infty}
\left\|b(\cdot)-b_{B(x_0,t)}\right\|_{L^{\widetilde{\Phi}}(B)} \|f\|_{L^\Phi(B(x_0,t))}\frac{dt}{t^{n+1}}
\\
& \quad + \frac{1}{\Phi^{-1}\big(r^{-n}\big)} \int_{2r}^{\infty}|b_{B}-b_{B(x_0,t)}|
\|f\|_{L^\Phi(B(x_0,t))}\Phi^{-1}\big(t^{-n}\big)\frac{dt}{t}
\\
& \lesssim \frac{\|b\|_{*}}{\Phi^{-1}\big(r^{-n}\big)}
\int_{2r}^{\infty}\Big(1+\ln \frac{t}{r}\Big)
\|f\|_{L^\Phi(B(x_0,t))}\Phi^{-1}\big(t^{-n}\big)\frac{dt}{t}.
\end{align*}
Summing $\left\|A\right\|_{L^\Phi(B)}$ and $\left\|B\right\|_{L^\Phi(B)}$, we have
\begin{equation*} \label{deckfV}
\|[b,G_{\alpha}]f_2\|_{L^\Phi(B)}
\lesssim \frac{\|b\|_{*}}{\Phi^{-1}\big(r^{-n}\big)}
\int_{2r}^{\infty}\Big(1+\ln \frac{t}{r}\Big) \|f\|_{L^\Phi(B(x_0,t))}\Phi^{-1}\big(t^{-n}\big)\frac{dt}{t}.
\end{equation*}
Finally, we obtain
\begin{align*}
\|[b,G_{\alpha}]f\|_{L^\Phi(B)} & \lesssim \|b\|_{*}\,\|f\|_{L^\Phi(2B)}
\\
& + \frac{\|b\|_{*}}{\Phi^{-1}\big(r^{-n}\big)} \int_{2r}^{\infty}\Big(1+\ln \frac{t}{r}\Big) \|f\|_{L^\Phi(B(x_0,t))}\Phi^{-1}\big(t^{-n}\big)\frac{dt}{t},
\end{align*}
and the statement of Lemma \ref{lem5.1.} follows by \eqref{ves2fd}.
\end{proof}

\begin{thm}\label{3.4.XcomT}
Let $b\in BMO$, $\a\in(0,1]$, $\Phi$ be a Young function which is lower type $p_0$ and upper type $p_1$ with $1<p_0\le p_1<\i$ and $(\varphi_1,\varphi_2)$ and $\Phi$ satisfy the condition
\begin{equation}\label{eq3.6.VZfrMaxcom}
\int_{r}^{\i}\Big(1+\ln \frac{t}{r}\Big) \, \Big(\es_{t<s<\i}\frac{\varphi_1(x,s)}{\Phi^{-1}\big(s^{-n}\big)}\Big) \, \Phi^{-1}\big(t^{-n}\big)\frac{dt}{t}  \le C \, \varphi_2(x,r),
\end{equation}
where $C$ does not depend on $x$ and $r$. Then  the operator $[b,G_{\a}]$ is bounded from $M^{\Phi,\varphi_1}(\Rn)$ to $M^{\Phi,\varphi_2}(\Rn)$.
\end{thm}

\begin{proof}
The statement of Theorem \ref{3.4.XcomT} follows by Lemma \ref{lem5.1.} and Theorem \ref{thm3.2.} in the same manner as in the proof of Theorem \ref{thm4.4.}.
\end{proof}

In \cite{Wilson1}, the author proved that the functions $G_{\alpha}f$ and $g_{\alpha}f$ are pointwise comparable. Thus, as a consequence of Theorem \ref{thm4.4.}, we have the following result.
\begin{cor}\label{wucor1.6}
Let $b\in BMO$, $\a\in(0,1]$, $\Phi$ be a Young function which is lower type $p_0$ and upper type $p_1$ with $1<p_0\le p_1<\i$ and the functions $(\varphi_1,\varphi_2)$ and $\Phi$ satisfy the condition \eqref{eq3.6.VZfrMaxcom}, then $[b,g_{\a}]$ is bounded from $M^{\Phi,\varphi_1}(\Rn)$ to $M^{\Phi,\varphi_2}(\Rn)$.
\end{cor}

By using the the argument as similar as the above proofs and that of Theorem \ref{thm4.4.x}, we can also show the boundedness of $[b,g_{\lambda,\alpha}^{*}]$.

\

\end{document}